\documentclass[10pt]{amsart}
\usepackage{amsmath}
\usepackage{amsfonts}
\usepackage{amsthm}
\usepackage{enumerate}
\usepackage{amssymb}
\usepackage{amscd}
\usepackage{mathrsfs}
\usepackage[all]{xy}

\DeclareMathOperator{\End}{End}
\DeclareMathOperator{\Ann}{Ann}
\DeclareMathOperator{\ann}{ann}
\DeclareMathOperator{\Dist}{Dist}
\DeclareMathOperator{\gr}{gr}
\DeclareMathOperator{\Lie}{Lie}
\DeclareMathOperator{\Gr}{Gr}
\DeclareMathOperator{\Spec}{Spec}
\DeclareMathOperator{\Ch}{Ch}
\DeclareMathOperator{\op}{op}
\DeclareMathOperator{\ad}{ad}
\begin{document}
\theoremstyle{definition}
\newtheorem*{examps}{Examples}
\newtheorem*{defn}{Definition}
\newtheorem*{thma}{Theorem A}
\newtheorem*{thmb}{Theorem B}
\newtheorem*{thmc}{Theorem C}
\newtheorem*{qna}{Question A}
\newtheorem*{qnb}{Question B}
\theoremstyle{remark}
\newtheorem*{rmk}{Remark}
\newtheorem*{rmks}{Remarks}
\theoremstyle{plain}
\newtheorem*{lem}{Lemma}
\newtheorem*{prop}{Proposition}
\newtheorem*{thm}{Theorem}
\newtheorem*{example}{Example}
\newtheorem*{examples}{Examples}
\newtheorem*{cor}{Corollary}
\newtheorem*{conj}{Conjecture}
\newtheorem*{hyp}{Hypothesis}
\newtheorem*{thrm}{Theorem}
\newtheorem*{quest}{Question}
\theoremstyle{remark}

\newcommand{\Grass}{\mathcal{G}}
\newcommand{\Zp}{{\mathbb{Z}_p}}
\newcommand{\Qp}{{\mathbb{Q}_p}}
\newcommand{\Fp}{{\mathbb{F}_p}}
\newcommand{\Fl}{\mathcal{B}}
\newcommand{\mbx}[1]{\quad\mbox{#1}\quad}
\newcommand{\invlim}{\lim\limits_{\longleftarrow}}
\newcommand{\Uhat}[1]{\widehat{U(\mathfrak{#1})_K}}
\newcommand{\U}[1]{U(\fr{#1})}
\newcommand{\N}{\mathcal{N}}
\renewcommand{\O}{\mathcal{O}}
\newcommand{\Sy}[1]{S(\fr{#1})}
\newcommand{\Syk}[1]{S(\fr{#1}_k)}
\newcommand{\huntGK}{\widehat{\U{g}^{\mb{G}}_{n,K}}}
\newcommand{\hugnK}{\widehat{\U{g}_{n,K}}}
\newcommand{\hugn}{\widehat{\U{g}_n}}
\newcommand{\Uk}[1]{U(\fr{#1}_k)}
\newcommand{\Uik}[2]{\Uk{#1}^{\mb{#2}_k}}
\newcommand{\Sik}[2]{\Syk{#1}^{\mb{#2}_k}}
\newcommand{\h}[1]{\widehat{#1}}
\newcommand{\fr}[1]{\mathfrak{{#1}}}
\newcommand{\mb}[1]{\mathbf{{#1}}}
\newcommand{\st}{\mid}
\newcommand{\be}{\begin{enumerate}[{(}a{)}]}
\newcommand{\ee}{\end{enumerate}}
\let\le=\leqslant  \let\leq=\leqslant
\let\ge=\geqslant  \let\geq=\geqslant

\title[Verma modules for Iwasawa algebras are faithful]{Verma modules for Iwasawa algebras \\ are faithful}
\author{Konstantin Ardakov}
\address{Mathematical Institute, University of Oxford, Oxford OX2 6GG}
\author{Simon Wadsley}
\address{Homerton College, Cambridge, CB2 8PQ}
\begin{abstract} We establish the faithfulness of Verma modules for rational Iwasawa algebras of split semisimple compact $L$-analytic groups. We also prove the algebraic independence of Arens-Michael envelopes over Iwasawa algebras and compute the centre of affinoid enveloping algebras of semisimple $p$-adic Lie algebras.
\end{abstract}
\subjclass[2010]{16S99, 22E50, 16D25}
\dedicatory{To Peter Schneider, on the occasion of his sixtieth birthday}
\maketitle
\vspace{-0.57cm}
\section{Introduction}

\subsection{Prime ideals in Iwasawa algebras} The majority of work so far related to the study of the prime spectrum of non-commutative Iwasawa algebras has produced negative results. By this we mean that results in this area have tended to put constraints on the set of prime ideals for such a ring rather than uncover prime ideals that were not known before: see for example \cite{VenCrelle}, \cite{Wad2007}, \cite{ArdMPCPS}, \cite{AWZ}, \cite{ArdWad2009}, \cite{Ard2012}. This work continues in that tradition. However, most of these theorems were established in characteristic $p$ in the first instance with eventual consequences in characteristic zero; by contrast, our methods here have a definite characteristic zero flavour and our results do not have immediate implications in positive characteristic.

Suppose that $L$ is a finite extension of $\Qp$ and that $K$ is a complete discretely valued field extension of $L$. Let $G$ be a compact open subgroup of the group of $L$-points of a connected, simply connected, split semisimple affine algebraic group scheme $\mb{G}$ defined over $\mathcal{O}_L$, and write $KG$ to denote the Iwasawa algebra of continuous $K$-valued distributions on $G$. The annihilator of every simple $KG$-module that is finite dimensional over $K$ is a prime ideal of finite codimension in $KG$, and moreover every prime ideal with this property will arise in this way. Evidence so far suggests that non-zero prime ideals in $KG$ that do not arise in this way are very scarce; indeed we suspect that when the algebraic group scheme $\mb{G}$ is simple and $G$ has trivial centre they do not exist. We present further evidence in that direction. 

\subsection{The main result} A natural place to look for more prime ideals in $KG$ is as annihilators of simple $KG$-modules that are not finite dimensional over $K$. By standard arguments in ring theory such ideals will always be prime and of infinite codimension in $KG$. Thus if our suspicion above is correct then all such annihilators must be zero. We show that this is the case for a large class of naturally arising examples. More precisely, in $\S \ref{PfA}$ below we prove

\begin{thma} Let $p$ be an odd very good prime for $\mb{G}$ and let $G$ be an open subgroup of $\mb{G}(\mathcal{O}_L)$ with trivial centre. Let $\mb{B}^+$ be a Borel subgroup scheme of $\mb{G}$ and let $K_\theta$ be a $1$-dimensional locally $L$-analytic $K$-representation of $B^+ := G \cap \mb{B}^+(\mathcal{O}_L)$. Then the induced $KG$-module $KG \otimes_{KB^+} K_\theta$ is faithful.
\end{thma}

We refer the reader to \cite[\S 6.8]{AW13} for a precise definition of what it means for a prime number $p$ to be a \emph{very good prime} for $\mb{G}$ and simply remark here that this condition is satisfied by any $p > 5$ if $\mb{G}$ is not of type $A$. These ``Verma modules" $KG \otimes_{KB^+} K_\theta$ are not always irreducible, but it follows from the work of Orlik and Strauch \cite[Theorem 3.5.2]{OrlikStrauch} that generically they \emph{are} irreducible at least when $L = \Qp$.

We note that Theorem A refutes the main result in \cite{Harris} whose proof has been known to contain a gap for a number of years, and whose statement was already known to be false for open pro-$p$  subgroups $G$ of $SL_2(\Zp)$ following the work of Wei, Zhang and the first author \cite{ArdMPCPS}, \cite{AWZ}, using very different methods to those found in this paper.  

\subsection{Two related results} We also prove some other results of independent interest. Write $D(H,K)$ to denote the algebra of $L$-locally analytic $K$-valued distributions on a compact $L$-analytic group $H$ in the sense of Schneider and Teitelbaum \cite{ST4}. There is a natural map from the Iwasawa algebra $KH$ to $D(H,K)$ because every $L$-locally analytic function on $H$ is continuous. We may also consider the subalgebra $D(H,K)_1$ consisting of those distributions in $D(H,K)$ that are supported at the identity in the sense of \cite{Kohlhaase}. At the end of $\S\ref{AlgInd}$ we prove the following result, which is essential to our proof of Theorem A.

\begin{thmb} The natural map $KH\otimes_K D(H,K)_1\to D(H,K)$ is an injection. \end{thmb} 

We note that in the case $H=\Zp$ an immediate consequence of Theorem B is the well-known algebraic independence of the logarithmic series $\log(1 + T)$ over the Iwasawa algebra $\O_K[[T]]$, so Theorem B may be viewed as a (slightly stronger) non-commutative analogue of this algebraic independence.

In \cite[\S 9.3]{AW13} we promised a future proof that the centre of the affinoid enveloping algebra $\hugnK$ is the closure of the image of the centre of $U(\mathfrak{g}_K)$ in $\hugnK$:

\begin{thmc} We have $Z\left(\hugnK\right) = \huntGK$.
\end{thmc} 

We provide a proof of Theorem C in $\S\ref{ZPi}$ of this paper, which is much simpler than that found in \cite{ArdThesis} for the case $n=0$. 

\subsection{Future work} We believe that this work raises two interesting questions. By an \emph{affinoid highest weight module} we mean a module that can be written in the form $\hugnK\otimes_{U(\fr{g})}M$ for some highest weight $U(\fr{g}_K)$-module $M$. Recall that an ideal that arises as the annihilator of a simple module is called \emph{primitive}. 

\begin{qna} Is it the case that every primitive ideal of $\hugnK$ with $K$-rational infinitesimal central character is the annihilator of a simple affinoid highest weight module? 
\end{qna}
Some evidence pointing towards a positive answer to Question A is provided by Duflo's main theorem in \cite{Duflo} that states that every primitive ideal of the classical enveloping algebra $U(\fr{g}_K)$ with $K$-rational infinitesimal central character is the annihilator of a highest weight module. In particular to answer yes, it would suffice to prove that every primitive ideal of $\hugnK$ is controlled by $U(\fr{g}_K)$. 

\begin{qnb}
Is every affinoid highest weight module that is not finite-dimensional over $K$ faithful as a $KG$-module?
\end{qnb}

Since Verma modules for classical enveloping algebras are generically irreducible, our Theorem \ref{AffKGfaithful} below may be viewed as giving evidence towards a positive answer to Question B. We believe that if we could give positive answers to both these questions then, in the case $L=\Qp$, we would be able to use the faithful flatness of $D(G,K)$ over $KG$ due to Schneider and Teitelbaum \cite{ST} together with our affinoid version \cite[Theorem D]{AW13}  of Quillen's Lemma to prove that every non-zero prime ideal of $KG$ is the annihilator of a finite dimensional simple module. 

\subsection{Acknowledgements}
The first author would like James Zhang for the invitation to visit Seattle in 2012, and for his encouragement that led to the eventual proof of Theorem \ref{AffKGfaithful}.  Many of the results in this paper were established during the Banff 2013 Workshop ``Applications of Iwasawa Algebras"; we thank its organisers for the invitation to visit and for the opportunity to speak about this work there. The second author thanks Homerton College for funding his travel to this workshop.

\section{Generalities on completed group rings}

\subsection{Module algebras and smash products}\label{ModAlg}
Let $k$ be a commutative base ring. Recall \cite[Chapter 4]{Mont} that if $H$ is a Hopf algebra over $k$ and $A$ is a $k$-algebra, then $A$ is a \emph{left $H$-module algebra} if there exists an action
\[ H \otimes_k A \to A, \quad r \otimes a \mapsto r \cdot a\]
such that
\[ r \cdot(ab) = (r_1\cdot a)(r_2\cdot b), \quad r \cdot 1 = \epsilon(r)1, \quad (rs)\cdot a = r\cdot(s \cdot a)\mbx{and} 1\cdot a = a\]
for all $r,s \in H$ and $a,b \in A$. Here we use the sumless Sweedler notation. There is a similar notion of \emph{right} $H$-module algebra, and the two notions coincide in the case when $H$ is commutative.

\noindent Whenever $A$ is a left $H$-module algebra, define $A \# H := A \otimes_k H$ and write $a\# r$ for the tensor $a \otimes r$ in $A \# H$. Then $A \# H$ becomes an associative $k$-algebra called the \emph{smash product} of $A$ with $H$, with multiplication given by 
\[ (a \# r)(b \# s) = a (r_1\cdot b) \# r_2 s \mbx{for all} a,b \in A \mbx{and} r,s \in H.\]
This smash product contains $A$ and $H$ as $k$-subalgebras, and $A$ is naturally a left $A \# H$-module via the rule
\[ (a \# r) \cdot b = a (r \cdot b) \mbx{for all} a, b \in A \mbx{and} r \in H.\]

\noindent Note that the subset of $H$-invariants in $A$, namely
\[ A^H := \{ a \in A : r \cdot a = \epsilon(r)a \mbx{for all} r \in H\}\]
is always a $k$-subalgebra of $A$. We have the following well-known

\begin{lem} Let $H$ be a Hopf algebra over $k$ and let $A$ be a left $H$-module algebra. Then 
\be 
\item $A$ is an $A \# H$---$A^H$-bimodule, and
\item $\End_{A \# H} A = (A^H)^{\op}$.
\ee
\end{lem}
\begin{proof} (a) The left regular representation of $A$ on itself commutes with the right regular representation, so we have to check that every $r \in H$ acts on $A$ by a right $A^H$-module endomorphism:
\[ r \cdot(ab) = (r\cdot a)b \mbx{for all} r \in H, a \in A \mbx{and} b \in A^H.\]
Now $r \cdot (ab) = (r_1 \cdot a)(r_2 \cdot b) = (r_1 \cdot a)( \epsilon(r_2) b ) = ( (r_1 \epsilon(r_2)) \cdot a) b$. But
$r_1 \epsilon(r_2) = r$ by the counit axiom in $H$. Therefore $r\cdot (ab) = (r\cdot a)b$ as required.

(b) Let $\varphi : A \to A$ be a left $A \# H$-module endomorphism. Since $\varphi$ is $H$-linear,
\[ r \cdot \varphi(1) = \varphi(r \cdot 1) = \varphi( \epsilon(r)1 ) = \epsilon(r) \cdot \varphi(1) \mbx{for all} r \in H,\]
which shows that $\varphi(1) \in A^H$. Since $\varphi$ is left $A$-linear, 
\[\varphi(x) = \varphi(x1) = x \varphi(1) \mbx{for all} x \in A\]
and therefore $\varphi$ agrees with right muliplication by $\varphi(1) \in A^H$. Hence the anti-homomorphism $A^H \to \End_{A \# H} A$ which sends $r \in A^H$ to right multiplication by $r$ is a bijection.
\end{proof}

\subsection{Locally constant functions}
Let $G$ be a profinite group. Recall that a function $f : G \to k$ is \emph{locally constant} if for all $g \in G$ there is an open neighbourhood $U$ of $g$ such that $f$ is constant on $U$.

\begin{defn} Let $C^\infty = C^\infty(G, k)$ denote the set of all locally constant functions from $G$ to $k$.
\end{defn}

$C^\infty$ becomes a unital commutative $k$-algebra when equipped with pointwise multiplication of functions. Moreover it is a Hopf algebra over $k$, with comultiplication $\Delta$, antipode $S$ and counit $\epsilon$ given by the formulas
\[\Delta(f)(g,h) = f(gh), \quad S(f)(g) = f(g^{-1}) \mbx{and} \epsilon(f) = f(1)\]
for all $f \in C^\infty$ and all $g,h \in G$.

\subsection{$G$-graded algebras}
\label{Ggraded}
We recall \cite[Definition 2.5]{Ard2012CONM} for the convenience of the reader.

\begin{defn} Let $G$ be a profinite group and let $A$ be a $k$-algebra. We say that $A$ is \emph{$G$-graded} if for each clopen subset $U$ of $G$ there exists a $k$-submodule $A_U$ of $A$ such that
\begin{enumerate}[{(}i{)}]
\item $A = A_{U_1} \oplus A_{U_2}\oplus \cdots \oplus A_{U_n}$ if $G = U_1 \cup \cdots \cup U_n$ is an open partition of $G$,
\item $A_U \leq A_V$ if $U \subseteq V$ are clopen subsets of $G$,
\item $A_U \cdot A_V \subseteq A_{UV}$ if $U, V$ are clopen subsets of $G$,
\item $1 \in A_U$ whenever $U$ is an open subgroup of $G$.
\end{enumerate}
\end{defn}

In this situation, \cite[Proposition 2.5]{Ard2012CONM} asserts that for a profinite group $G$, a $k$-algebra $A$ is $G$-graded if and only if $A$ is a $C^\infty$-module algebra.

For every open subgroup $U$ of $G$, let ${^UC^\infty}$ denote the $k$-subalgebra of functions $f \in C^\infty$ that are constant on the left cosets $gU$ of $U$ in $G$. Then 
\[C^\infty = \bigcup\limits_{U \leq_o G} {^UC^\infty}\]
and ${^UC^\infty}$ is even a Hopf subalgebra of $C^\infty$ whenever $U$ is normal, because it is naturally isomorphic to the algebra of $k$-valued functions on the finite group $G/U$ in this case.

\begin{prop} Let $G$ be a profinite group, let $U$ be an open normal subgroup of $G$ and let $A$ be a $G$-graded $k$-algebra. Then the algebra of ${^UC^\infty}$-invariants of $A$ is precisely $A_U$.
\end{prop}
\begin{proof} The Hopf algebra ${^UC^\infty}$ is spanned by the characteristic functions $\delta_{gU}$ of all the cosets of $U$ in $G$. By the construction given in the proof of \cite[Proposition 2.5]{Ard2012CONM}, the action of ${^UC^\infty}$ on $A$ and the $G$-graded structure are related by 
\[A_{gU} = \delta_{gU}\cdot A \mbx{for all} g \in G.\] 
Now the $\delta_{gU}$ form a family of commuting idempotents in ${^UC^\infty}$, and 
\[\epsilon(\delta_{gU}) = \delta_{gU}(1) = \left\{ \begin{array}{l} 1 \mbx{if} g \in U \\ 0 \mbx{if} g \notin U\end{array}\right.\]
so that $( \delta_{gU} - \epsilon(\delta_{gU}) ) \delta_U = 0$ for all $g \in G$. This implies that
\[ A_U \subseteq A^{{^UC^\infty}}.\]
On the other hand, let $\mu \in A$ be ${^UC^\infty}$-invariant. Then
\[\mu = (\delta_U \cdot \mu) + (1 - \delta_U) \cdot \mu = (\delta_U \cdot \mu) + \epsilon(1 - \delta_U) \mu = \delta_U \cdot \mu \in A_U\]
as required.
\end{proof}

\begin{cor} The algebra of $C^\infty$-invariants in $A$ is $\bigcap_{U \leq_o G} A_U$.
\end{cor}

\subsection{Completed group rings}\label{CompGpRng}
Let $k G$ denote the \emph{completed group ring} of $G$ with coefficients in $k$:
\[ k G := \invlim k[G/U]\]
where the inverse limit is taken over all the open normal subgroups $U$ of $G$. 

\begin{lem} Let $G$ be a profinite group and let $S \subseteq k$ be a multiplicatively closed subset consisting of non zero-divisors.
\be
\item $kG$ is a $G$-graded $k$-algebra.
\item The algebra of $C^\infty$-invariants in $kG$ is $k$.
\item The central localisation $S^{-1} kG$ of $kG$ is a $G$-graded $S^{-1}k$-algebra.
\item The algebra of $C^\infty$-invariants in $S^{-1} kG$ is $S^{-1}k$.
\ee
\end{lem}
\begin{proof} (a) This follows from \cite[Lemma 2.9]{Ard2012CONM}.

(b) For any open normal subgroup $U$ of $G$ let $\epsilon_U : kG \twoheadrightarrow k[G/U]$ be the natural surjection. Let $x \in kG \backslash k$; then by the definition of inverse limit, we can find some open subgroup $U$ of $G$ such that the $\epsilon_U(x) \notin \epsilon_U(k)$. But $\epsilon_U(k) = \epsilon_U(kU)$ so $x \notin kU$. Hence $\bigcap_{U \leq_o G} kU = k$ and we may apply Corollary \ref{Ggraded}.

(c) By (a) we can find we family $((kG)_U)$ (for $U$ ranging over the clopen subsets of $G$) of $k$-submodules of $kG$ satisfying the conditions of the definition of a $G$-graded $k$-algebra. Then the family $(S^{-1}(kG)_U)$ gives $S^{-1}kG$ the structure of a $G$-graded $k$-algebra.

(d) In view of part (b), it is enough to prove that $S^{-1} kU \cap kG = kU$ for every open normal subgroup $U$ of $G$. Suppose that $s^{-1}x \in kG$ for some $s \in S$ and $x \in kU$. Let $y = \delta_U \cdot (s^{-1}x) \in kU$ and $z = (1 - \delta_U) \cdot s^{-1}x \in (1 - \delta_U) \cdot kG$ so that $s^{-1}x = y + z$. Then $x - sy \in kU$ and $sz \in (1 - \delta_U) \cdot kG$, so
\[ x - sy = sz \in \delta_U \cdot kG \quad \cap \quad (1 - \delta_U) \cdot kG = 0\]
and therefore $s^{-1}x = y \in kU$ as required. \end{proof}

\section{The multiplication map $KG \otimes_K D(G,K)_1 \to D(G,K)$ is injective}\label{AlgInd}

\subsection{Compact $p$-adic analytic groups}\label{SimpleKG} Now let $G$ be a compact $p$-adic analytic group and let $R$ be a complete discrete valuation ring of characteristic zero with a residue field $k$ of characteristic $p$. Fix a uniformiser $\pi \in R$ and let $K$ be the field of fractions of $R$. We define the \emph{algebra of continuous $K$-valued distributions} of $G$ to be the central localisation
\[ KG := K \otimes_R RG.\]
of the completed group ring $RG$. In this situation we may naturally form three smash product algebras following $\S \ref{ModAlg}$:
\begin{itemize}
\item $\mathcal{A}_k = kG \# C^\infty(G,k)$,
\item $\mathcal{A}_R = RG \# C^\infty(G,R)$, and
\item $\mathcal{A}_K = KG \# C^\infty(G,K)$.
\end{itemize}
Then $\mathcal{A}_T$ naturally acts on $TG$ for all $T \in \{K,R,k\}$ by Lemma \ref{CompGpRng}.
\begin{thm} 
\be \item $kG$ is a simple $\mathcal{A}_k$-module.
\item $KG$ is a simple $\mathcal{A}_K$-module.
\ee
\end{thm}
\begin{proof} (a) Since $\mathcal{A}_k$ is generated by $kG$ and $C^\infty(G,k)$, an $\mathcal{A}_k$-submodule of $kG$ is just a left ideal $I$ of $kG$ such that 
\[ C^\infty(G,k) \cdot I \subseteq I.\]
By \cite[Definition 2.6]{Ard2012CONM}, we see that every open subgroup $U$ of $G$ controls $I$. Hence $I^\chi$, the controller subgroup of $I$, is trivial. Now \cite[Theorem A]{Ard2012CONM} is also valid for left ideals, and in our situation this implies that the left ideal $I$ is generated as a left ideal by its intersection with the ground field $k$. Therefore $I = 0$ or $I = kG$.

(b) Let $I$ be a $\mathcal{A}_K$-submodule of $KG$. Then $I \cap RG$ is a $\mathcal{A}_R$-submodule of $RG$ and $((I \cap RG) + \pi RG) / \pi RG$ is a $\mathcal{A}_k$-submodule of $kG$. By part (a), we see that either $(I \cap RG) + \pi RG = RG$ or $(I \cap RG) + \pi RG = \pi RG$. In the first case, the $\pi$-adic completeness of $RG$ implies that $I \cap RG = RG$ and in the second case, $I \cap RG \subseteq \pi RG$. Thus in the first case $I=KG$. In the second case, since $I = \pi I$, an easy induction shows that $I \cap RG \subseteq \pi^n RG$ for all $n \geq 0$ and so $I \cap RG = 0$, therefore $I = 0$.
\end{proof}

\subsection{Theorem}\label{LinIndep}
Let $G$ be a compact $p$-adic analytic group and let $KG \to \mathcal{D}$ be a homomorphism of $C^\infty(G,K)$-module algebras. Let $\mathcal{D}_1$ denote the algebra of $C^\infty(G,K)$-invariants in $\mathcal{D}$. Then the multiplication map
\[ KG \otimes_K \mathcal{D}_1 \longrightarrow \mathcal{D} \]
is injective.
\begin{proof} Let $\alpha_1, \cdots, \alpha_m \in KG$ be linearly independent over $K$ and let $t_1, \ldots, t_m \in \mathcal{D}_1$ be given such that
\[ \alpha_1 t_1 + \cdots + \alpha_m t_m = 0\]
inside $\mathcal{D}$. The $\mathcal{A}_K$-module $KG$ is simple by Theorem \ref{SimpleKG}(b) and its endomorphism ring $\End_{\mathcal{A}_K}(KG)$ seen to be $K$ by Lemma \ref{ModAlg}(b) and Lemma \ref{CompGpRng}(d). It follows that the $\alpha_i$ are linearly independent over $\End_{\mathcal{A}_K}(KG)$, and so using the Jacobson Density Theorem we can find $\xi_1,\ldots,\xi_m \in \mathcal{A}_K$ such that
\[ \xi_i(\alpha_j) = \delta_{ij} \mbx{for all} j=1,\ldots,m.\]
Now $\mathcal{D}$ is a left $KG$-module by left multiplication, and this action commutes with right multiplication by $\mathcal{D}_1$. Consequently, $\mathcal{D}$ is an $\mathcal{A}_K$--- $\mathcal{D}_1$-bimodule. Therefore
\[ 0 = \xi_i \cdot \left( \sum_{j=1}^m \alpha_j t_j \right) = \sum_{j=1}^m (\xi_i \cdot \alpha_j) t_j = \sum_{j=1}^m \delta_{ij} t_j = t_i\]
for all $i = 1, \ldots, m$.
\end{proof}
\subsection{Locally analytic distribution algebras}
\label{LocAn}
Now suppose that $L$ is a finite extension of $\Qp$ that is contained in $K$ and let $M$ be a locally $L$-analytic manifold. The \emph{space of $K$-valued $L$-analytic distributions} $D(M,K)$ on $M$ is the strong dual $C^{an}(M,K)'_b$ of the space $C^{an}(M,K)$ of $K$-valued locally $L$-analytic functions on $M$ --- see \cite[\S 2]{ST4}.

When $G$ is a locally $L$-analytic group, multiplication in the group $G$ induces a structure of a unital associative $K$-algebra on $D(G,K)$ \cite[Proposition 2.3]{ST4}. This algebra is called the \emph{algebra of $K$-valued locally $L$-analytic distributions} on $G$.

\begin{lem} $D(G,K)$ is a $G$-graded $K$-algebra, whenever $G$ is compact.
\end{lem}
\begin{proof} Since $G$ is a locally $L$-analytic group, every clopen subset $U$ of $G$ is a locally $L$-analytic manifold, so we may set
\[ D(G,K)_U := D(U,K).\]
With this definition, parts (ii) and (iv) of Definition \ref{Ggraded} are clear. We may assume that all clopen subsets featuring in the statement of parts (i) and (iii) of the Definition are finite unions of cosets of a fixed open normal subgroup $H$ of $G$. For each $g \in G$ let $\delta_g \in D(G,K)$ be the Dirac distribution. It was observed in the proof of \cite[Lemma 3.1]{ST4} that
\[ D(G,K) = \bigoplus_{g \in G/H} \delta_g \ast D(H,K).\]
Part (i) follows immediately, and part (iii) follows since $D(H,K)$ is a subalgebra of $G$ which is stable under conjugation by each $\delta_g$ inside $D(G,K)$.
\end{proof}

\begin{cor} Let $G$ be a compact $L$-analytic group. Then
\be
 \item $D(G,K)$ is a $C^\infty$-module algebra. 
 \item The algebra of invariants under this action is precisely 
 \[D(G,K)_1 := \bigcap_{H \leq_o G} D(H,K).\]
 \item The natural map $KG \otimes_K D(G,K)_1 \to D(G,K)$ is injective.
 \item Let $\fr{g}_K := K \otimes_{L} \mathcal{L}(G)$. Then the natural map 
 \[KG \otimes_K U(\fr{g}_K) \to D(G,K)\]
 is also injective.
\ee 
\end{cor}
\begin{proof} (a) Apply Lemma \ref{LocAn} together with \cite[Proposition 2.5]{Ard2012CONM}. 

(b) Apply part (a) together with Corollary \ref{Ggraded}.

(c) This follows from Theorem \ref{LinIndep}.

(d) It was observed in \cite[\S 2]{ST4} that $U(\fr{g}_K)$ is contained in $D(H,K)$ for every open subgroup $H$ of $G$; therefore $U(\fr{g}_K) \subseteq D(G,K)_1$. Now apply part (c).
\end{proof}
We remark that it follows from the work of Kohlhaase \cite[Proposition 1.2.8]{Kohlhaase} that the image of $U(\fr{g}_K)$ is in fact dense in $D(G,K)_1$. It is the \emph{hyper-enveloping algebra} or \emph{Arens-Michael envelope} of $U(\fr{g}_K)$ in the sense of Schmidt \cite{Schmidt3}.

\section{Affinoid enveloping algebras and Verma modules}
\label{AffEnvAlg}
\subsection{The adjoint action of $\mb{G}(R)$ on $U(\fr{g})$}
\label{AdU}
Let $\mb{G}$ be a connected, split semisimple, affine algebraic group scheme over $R$ with Lie algebra $\fr{g}$. The Lie algebra $\fr{g}$ is a $\mb{G}$-module via the adjoint action; see \cite[II.1.12(1), I.7.18]{Jantzen}. In particular, the group of $R$-points $\mb{G}(R)$ of $\mb{G}$ acts on $\fr{g}$ by Lie algebra automorphisms, and therefore by functoriality on $U(\fr{g})$ by $R$-algebra automorphisms. This action preserves the natural PBW-filtration 
\[ 0 \subset F_0 U(\fr{g}) \subset F_1 U(\fr{g}) \subset \cdots \]
on $U(\fr{g})$. Let $\Phi$ be the root system of $\mb{G}$ relative to a fixed split maximal torus $\mb{T}$, and let $x_\alpha : \mb{G}_a \to \mb{G}$ and $e_\alpha = (dx_\alpha)(1) \in \fr{g}$ be the root homomorphism and root vector corresponding to the root $\alpha \in \Phi$, respectively. 

\begin{lem} Let $r \in R$, $\alpha \in \Phi$.
\be
\item For every $\mb{G}$-module $M$, each divided power $\frac{e_\alpha^m}{m!}$ preserves $M$.
\item For all $b \in U(\fr{g})$ there exists $i \geq 1$ such that $\frac{\ad(re_\alpha)^i}{i!} \cdot b = 0$.
\item $x_\alpha(r)\cdot a = \sum_{m = 0}^\infty \frac{\ad(re_\alpha)^m}{m!}(a)$ for all $a \in U(\fr{g})$.
\ee
\end{lem}
\begin{proof} (a) We may view $M$ as a $\mb{G}_a$-module by restriction via $x_\alpha$. Hence it is a module over the \emph{distribution algebra} $\Dist(\mb{G}_a)$ of the additive group $\mb{G}_a$, by \cite[I.7.11]{Jantzen}. It is known \cite[I.7.3, I.7.8]{Jantzen} that this distribution algebra has a basis consisting of the divided powers of the generator of $\Lie(\mb{G}_a)$.

(b) $U(\fr{g})$ is a $\mb{G}$-module so $\frac{\ad(r e_\alpha)^i}{i!}(b)$ lies in $U(\fr{g})$ by part (a). Now $[\fr{g}, F_j U(\fr{g})] \subseteq F_{j-1} U(\fr{g})$ for all $j \geq 0$, so if $b \in F_{i-1}U(\fr{g})$ for some $i \geq 1$ then $\ad(r e_\alpha)^i (b) = 0$. The result follows because $U(\fr{g})$ has no $R$-torsion.

(c) This follows from the definitions --- see \cite[I.2.8(1), I.7.12]{Jantzen}. Note that the right hand side of the equation makes sense by part (b).
\end{proof}

\subsection{Deformations and $\pi$-adic completions}
\label{Defs}
Recall \cite[\S 3.5]{AW13} that a \emph{deformable $R$-algebra} is a positively $\mathbb{Z}$-filtered $R$-algebra $A$ such that $F_0A$ is an $R$-subalgebra of $A$ and $\gr A$ is a flat $R$-module. A \emph{morphism} of deformable $R$-algebras is an $R$-linear filtered ring homomorphism. Let $A$ be a deformable $R$-algebra. Its \emph{$n$-th deformation} is the $R$-subalgebra
\[A_n := \sum_{i\geq 0} \pi^{in} F_iA \subseteq A.\]
$A_n$ becomes a deformable $R$-algebra when we equip $A_n$ with the subspace filtration arising from the given filtration on $A$, and multiplication by $\pi^{in}$ on graded pieces of degree $i$ extends to a natural isomorphism of graded $R$-algebras 
\[\gr A \stackrel{\cong}{\longrightarrow} \gr A_n\]
by \cite[Lemma 3.5]{AW13}. The assignment $A \mapsto A_n$ is functorial in $A$. $\h{A} := \invlim A / \pi^a A$ will denote the $\pi$-adic completion of $A$.  Recall \emph{almost commutative affinoid $K$-algebras} from \cite[\S 3.8]{AW13}. Such an algebra $B$ has a \emph{double associated graded ring} $\Gr(B)$; when 
\[B = \h{A_{n,K}} := \h{A_n} \otimes_RK \] 
for some deformable $R$-algebra $A$, \cite[Corollary 3.7]{AW13} tells us that $\Gr(B)$ can be computed as follows:
\[\Gr(B) = \Gr(\h{A_{n,K}}) \cong \gr A / \pi \gr A.\]

We fix the deformation parameter $n$ in what follows.

\subsection{Semisimple affinoid enveloping algebras}
\label{AdjointAction}
The enveloping algebra $\U{g}$ has associated graded ring $\gr \U{g} = \Sy{g}$ and is therefore deformable. For each $n \geq 0$, the \emph{semisimple affinoid enveloping algebra}
\[ \hugnK \]
is almost commutative affinoid, and its double associated graded ring is
\[ \Gr(\hugnK) \cong \Syk{g}.\]
By functoriality, the adjoint action of $\mb{G}(R)$ on $U(\fr{g})$ from $\S \ref{AdU}$ extends to each $\hugnK$. This action preserves the double filtration on $\hugnK$ and induces an action of $\mb{G}(R)$ on $\Gr(\hugnK)  \cong \Syk{g}$, which factors through $\mb{G}(k)$.

\begin{prop} Let $r \in R$, $\alpha \in \Phi$ and $a \in \hugnK$. Then the series 
\[ \sum_{m\geq 0} \frac{ \ad( r e_\alpha )^m }{m!} (a) \]
converges in $\hugnK$ to $x_\alpha(r) \cdot a$.
\end{prop}
\begin{proof} Without loss of generality, we may assume that $a \in \hugn$. Let $D := \ad(re_\alpha)$, viewed as a derivation of $\hugnK$ and let $N \geq 0$ be given. Then there exists $b \in U(\fr{g})_n$ such that $a \equiv b \mod \pi^N \hugn$. Now $\frac{D^i}{i!}(b) = 0$ for some $i \geq 1$ by Lemma \ref{AdU}(b). Since $U(\fr{g})_n$ is also a $\mb{G}$-module, $\frac{D^i}{i!}$ preserves $\pi^N U(\fr{g})_n$ and therefore also $\pi^N \hugn$ by Lemma \ref{AdU}(a). Thus for all $N \geq 0$ there exists $i \geq 1$ such that
\[ \frac{D^i}{i!}(a) \in \pi^N \hugn.\]
Hence $\frac{D^i}{i!}(a) \to 0$ as $i \to \infty$ inside $\hugn$, and therefore the series $\sum_{m=0}^\infty \frac{D^m}{m!}(a)$ converges to an element $e^D(a)$ of $\hugn$, say. This defines an $R$-linear endomorphism $a \mapsto e^D(a)$ of $\hugn$, which agrees with the action of $x_\alpha(r)$ on its dense subalgebra $U(\fr{g})_n$ by Lemma \ref{AdU}(c).
\end{proof}

\begin{cor}
\be
\item Every two-sided ideal $I$ of $\hugnK$ is preserved by $\mb{G}(R)$.
\item Every central element $z$ of $\hugnK$ is fixed by $\mb{G}(R)$.
\ee
\end{cor}
\begin{proof} Since the ring $R$ is local, the Chevalley group $\mb{G}(R)$ is generated by elements of the form $x_\alpha(r)$ (where $r \in R$ and $\alpha \in \Phi$) by a result of Abe \cite[Proposition 1.6]{Abe69}. Fix $r \in R$ and $\alpha \in \Phi$.

(a) Let $a \in I$. Since $I$ is a two-sided ideal, $\sum_{m=0}^N \frac{\ad(r e_\alpha)^m}{m!}(a) \in I$ for all $N \geq 0$. This sequence of elements converges to $x_\alpha(r)\cdot a$ by the proposition, so $x_\alpha(r)\cdot a \in I$ because $I$ is \emph{closed} by \cite[Corollary I.5.5]{LVO}.

(b) By the proposition, we have $x_\alpha(r) \cdot z = \sum_{m=0}^\infty \frac{\ad(r e_\alpha)^m}{m!}(z) = z$ because $[re_\alpha, z] = 0$ by assumption.
\end{proof}

\subsection{The centre of $\hugnK$}
\label{ZPi}
We assume from now on that $\mb{G}$ is simply-connected. The ring of invariants $\U{g}^{\mb{G}}$ of $\U{g}$ under the adjoint action of $\mb{G}$ is a deformable $R$-algebra, and it was shown in \cite[Proposition 9.3(a)]{AW13} that the completion of its $n$-th deformation $\huntGK$ is contained in the centre $Z\left(\hugnK\right)$ of $\hugnK$. We can now prove Theorem C from the Introduction.

\begin{thm} We have $Z\left(\hugnK\right) = \huntGK$.
\end{thm}
\begin{proof} By base-changing to the completion of the maximal unramified extension of $K$, we may assume that the residue field $k$ of $R$ is algebraically closed. Let $z \in Z\left(\hugn\right)$. Then $z$ is fixed by the action of $\mb{G}(R)$ by Corollary \ref{AdjointAction}(b), so the symbol $\gr \bar{z}$ of the image $\bar{z}$ of $z$ in $\gr_0 \hugnK$ is fixed by the induced action of $\mb{G}(R)$ on $\Gr(\hugnK) \cong \Syk{g}$, which is just the adjoint action of $\mb{G}(k)$ on $\Syk{g}$. Since the group $\mb{G}_k$ is reduced and $k$ is algebraically closed, it follows from \cite[Remark I.2.8]{Jantzen} that $\gr \bar{z} \in \Sik{g}{G}$. Therefore
\[ \Gr(\huntGK) \subseteq \Gr(Z(\hugnK)) \subseteq \Sik{g}{G}.\]
It was shown in the proof of \cite[Theorem 6.10]{AW13} that the identification of $\Gr(\hugnK)$ with $\Syk{g}$ maps $\Gr(\huntGK)$ onto $\Sik{g}{G}$. Therefore the two inclusions displayed above are equalities and the result follows.
\end{proof}

\subsection{Affinoid Verma modules}
\label{AffVerma}

Let $\mb{T} \subset \mb{B}$ be a split maximal torus in $\mb{G}$ contained in a Borel subgroup $\mb{B}$. We will view the unipotent radical $\mb{N}$ of $\mb{B}$ as being generated by the \emph{negative} roots of the adjoint action of $\mb{T}$ on $\mb{G}$, and let $\mb{N}^+$ be the unipotent radical of the opposite Borel $\mb{B}^+$ containing $\mb{T}$. Let $\fr{t}, \fr{b}, \fr{n}$, $\fr{n}^+$ and $\fr{b}^+$ be the corresponding Lie algebras, so that we have the root space decomposition \[\fr{g} = \fr{n} \oplus \fr{t} \oplus \fr{n}^+.\]

Let $\lambda : \pi^n \fr{t} \to R$ be an $R$-linear character. View $\lambda$ as a character of $\pi^n \fr{b}^+$ by pulling back along the surjection $\pi^n \fr{b}^+ \twoheadrightarrow \pi^n \fr{t}$ with kernel $\pi^n \fr{n}^+$, and let $K_\lambda$ be the corresponding one-dimensional module over the affinoid enveloping algebra $\widehat{U(\fr{b}^+)_{n,K}}$.

\begin{defn} The \emph{affinoid Verma module} with highest weight $\lambda$ is 
\[ \widehat{V^\lambda} := \hugnK \otimes_{\widehat{U(\fr{b}^+)_{n,K}}} K_\lambda.\]
\end{defn}

We will compute the annihilator of this affinoid Verma module in $\hugnK$. To do this, we will first need to understand its \emph{characteristic variety} $\Ch(\widehat{V^\lambda})$. Recall from \cite[Definition 3.3]{AW13} that this is an algebraic subset of the prime spectrum of the polynomial algebra $\Gr(\hugnK) = \Syk{g}$.

\begin{lem} $\Ch(\widehat{V^\lambda}) = (\fr{b}^+_k)^\perp.$
\end{lem}
\begin{proof} Let $\bar{\lambda}:\fr{b}^+_k \to k$ be defined by $\bar{\lambda}(\bar{x}) = \overline{\lambda(\pi^n x)}$ for any $x \in \fr{b}^+$, and let $k_{\bar{\lambda}}$ denote the corresponding $S(\fr{b}_k^+)$-module. Then there is a natural double filtration on $\widehat{V^\lambda}$ such that 
\[\Gr(\widehat{V^\lambda}) = \Syk{g} \otimes_{S(\fr{b}_k^+)} k_{\bar{\lambda}}.\]
The support of this graded module inside $\Spec \Syk{g}$ is $(\fr{b}^+_k)^\perp$ by definition.
\end{proof}

\subsection{The annihilator of an affinoid Verma module}
\label{AnnVerma}

It is well known that the centre $Z(\fr{g}_K)$ of $U(\fr{g}_K)$ acts on the classical Verma module $V^\lambda := U(\fr{g}_K) \otimes_{U(\fr{b}^+_K)} K_\lambda$ by a character $\chi_\lambda : Z(\fr{g}_K) \to K$; see \cite[Proposition 7.4.4]{Dix}. Since $V^\lambda$ is dense in $\widehat{V^\lambda}$, the action $Z(\fr{g}_K)$ on $\widehat{V^\lambda}$ also factors through $\chi_\lambda$. 

In \cite[\S 9.8]{AW13} we defined the \emph{nilpotent cone} in $\fr{g}^\ast$ to be the set of zeros $\N^\ast$ of $\mb{G}_k$-invariant polynomials	in $\Syk{g} = \O(\fr{g}^\ast_k)$ with no constant term:
\[ \N^\ast = V( \Sik{g}{G}_+ ).\]

After the next preliminary result, we will be able to compute the annihilator of $\widehat{V^\lambda}$ and thereby prove a precise affinoid analogue of \cite[Theorem 8.4.3]{Dix}.

\begin{lem} Suppose that $k$ is algebraically closed, and $p$ is very good for $\mb{G}$. Then
\be
\item the ideal $\Gr( \ker \chi_\lambda \cdot \hugnK )$ equals $\Sik{g}{G}_+ \cdot \Syk{g}$.
\item This ideal is prime.
\item If $G := \mb{G}(k)$ then $G \cdot (\fr{b}_k^+)^\perp$ is Zariski dense in $\N^\ast$.
\ee 
\end{lem}
\begin{proof} Part (a) follows from \cite[Theorem 6.10]{AW13}, and part (b) follows from the proof of \cite[Proposition 3.4.1]{BMR1}. In the proof of \cite[Proposition 3.4.1]{BMR1} it was also shown that under our hypotheses, the natural action map $G \times^B (\fr{b}_k^+)^\perp \to \N^\ast$ induces an isomorphism $\O(\N^\ast) \to \O(G \times^B (\fr{b}_k^+)^\perp)$, and is therefore dominant. Part (c) follows.
\end{proof}

\begin{thm} If $p$ is very good for $\mathbf{G}$, then the annihilator of $\widehat{V^\lambda}$ inside $\widehat{U(\mathfrak{g})_{n,K}}$ is generated by $\ker \chi_\lambda$.
\end{thm}
\begin{proof} Let $I_\lambda$ be the annihilator of $\widehat{V^\lambda}$ inside $\mathcal{U}_n := \hugnK$ and let $J_\lambda =  \ker \chi_\lambda \cdot \mathcal{U}_n$. Then $J_ \lambda \subseteq I_\lambda$ by the remarks made at the start of $\S \ref{AnnVerma}$. By base-changing to the completion of the maximal unramified extension of $K$, we may assume that the residue field $k$ of $R$ is algebraically closed. This allows us to identify the characteristic varieties of finitely generated $\mathcal{U}_n$-modules with their corresponding sets of $k$-points. 

Since $J_\lambda \subseteq I_\lambda$, the characteristic variety $\Ch(\mathcal{U}_n / I_\lambda)$ is contained in $\Ch\left(\mathcal{U}_n / J_\lambda \right)$, which equals $\N^\ast$ by part (a) of the lemma.

The two-sided ideal $I_\lambda$ is $\mb{G}(R)$-stable by Corollary \ref{AdjointAction}(a), so $\Ch(\mathcal{U}_n / I_\lambda)$ is stable under the adjoint action of $G:=\mb{G}(k)$ on $\fr{g}^\ast_k$. Also $I_\lambda$ annihilates $\widehat{V^\lambda}$, so $\Ch(\mathcal{U}_n / I_\lambda)$ contains $\Ch(\widehat{V^\lambda})$ which is equal to $(\fr{b}_k^+)^\perp$ by Lemma \ref{AffVerma}. Therefore 	$\Ch(\mathcal{U}_n / I_\lambda)$ contains $G \cdot (\fr{b}_k^+)^\perp$ which is Zariski dense in $\N^\ast$ by part (c) of the lemma. Since $\Ch(\mathcal{U}_n / I_\lambda)$ is closed, we deduce that $\Ch(\mathcal{U}_n / I_\lambda) = \mathcal{N}^\ast = \Ch(\mathcal{U}_n / J_\lambda)$. Hence
\[\Gr(J_\lambda) \subseteq \Gr(I_\lambda) \subseteq \sqrt{\Gr(I_\lambda)} = \sqrt{\Gr(J_\lambda)} = \Gr(J_\lambda)\]
because $\Gr(J_\lambda)$ is prime by part (b) of the lemma. The result follows.
\end{proof}

\subsection{The action of $\widehat{U(\fr{t})_{n,K}}$ on affinoid Verma modules}
\label{AdRT}
In $\S\ref{FaithfulIwasawa}$ we will need the following elementary result about the analytic density of certain infinite discrete subsets in affinoid polydiscs. 

\begin{lem} Let $A_1, A_2, \ldots, A_\ell$ be infinite subsets of $R$ and suppose that an element $f$ in the Tate algebra $K \langle x_1,\ldots, x_\ell \rangle$ vanishes on $A_1 \times A_2 \times \cdots \times A_\ell$. Then $f = 0$.
\end{lem}
\begin{proof} Proceed by induction on $\ell$. The case when $\ell = 0$ is vacuous so we may assume that $\ell \geq 1$. For every $y \in A_\ell$ let $g_y(x_1,\ldots,x_{\ell-1}):=f(x_1,\ldots,x_{\ell-1},y) \in K\langle x_1,\ldots, x_{\ell - 1}\rangle$. Then $g_y$ vanishes on $A_1 \times A_2 \times \cdots \times A_{\ell - 1}$ so by induction, $g_y = 0$ for all $y \in A_\ell$. Therefore $x_\ell - y$ divides $f$ for all $y \in A_\ell$. 

Now $K\langle x_1,\ldots, x_\ell \rangle$ is a Noetherian unique factorisation domain by \cite[Theorem 3.2.1]{FvdPut}, and as $y$ ranges over $A_\ell$, the $x_\ell - y$ form a collection of infinitely many distinct irreducible elements of $K\langle x_1,\ldots, x_\ell \rangle$. Therefore $f = 0$.
\end{proof}

Now consider the action of $\widehat{U(\fr{t})_{n,K}}$ on the affinoid Verma module $\widehat{V^\lambda}$ from $\S \ref{AffVerma}$. Let $v_\lambda \in \widehat{V^\lambda}$ be a highest weight vector, let $\alpha_1,\cdots, \alpha_m \in \fr{t}_K^\ast$ be the positive roots corresponding to the adjoint action of $\fr{t}$ on $\fr{n}^+$ and choose a generator $f_i \in \fr{n}$ for the $-\alpha_i$-root $R$-submodule of $\fr{n}$. Write $\mathbf{f}^\beta := f_1^{\beta_1} f_2^{\beta_2} \cdots f_m^{\beta_m} \in U(\fr{n})$ for any $\beta \in \mathbb{N}^m$. It is easy to see that
\[ h \cdot \mathbf{f}^\beta v_\lambda = (\lambda - \sum_{j=1}^m \beta_j\alpha_j)(h) \mathbf{f}^\beta v_\lambda \quad\mbox{for all}\quad h \in \fr{t}.\]
Thus $\mathbf{f}^\beta v_\lambda$ spans a one-dimensional $U(\fr{t})_K$-submodule of $U(\fr{n}_K)\cdot v_\lambda \subset \widehat{V^\lambda}$, so $K \mathbf{f}^\beta$ is actually a $\widehat{U(\fr{t})_{n,K}}$-module where $\fr{t}_K$ acts via the character $\lambda - \sum_{j=1}^m \beta_j\alpha_j \in \fr{t}^\ast_K$. In particular, we see that $U(\fr{n}_K)\cdot v_\lambda$ is a locally finite $\widehat{U(\fr{t})_{n,K}}$-module.

\begin{prop} The action of $\widehat{U(\fr{t})_{n,K}}$ on $U(\fr{n}_K)\cdot v_\lambda$ is faithful.
\end{prop}
\begin{proof} Let $\alpha_1,\ldots,\alpha_\ell$ be the simple roots, let $h_1,\ldots, h_\ell \in \fr{t}$ be the corresponding coroots and let $\omega_1,\ldots, \omega_\ell \in \fr{t}^\ast_K$ be the corresponding system of fundamental weights, so that $\omega_i(h_j) = \delta_{ij}$ for all $i,j=1,\ldots, \ell$. By \cite[\S 13.1]{Hum}, we may use the Cartan matrix $C = (\langle \alpha_i,\alpha_j \rangle)$ associated to the root system of $\fr{g}_K$ to express the simple roots in terms of the fundamental weights:
\[\alpha_j = \sum_{k=1}^\ell C_{jk} \omega_k \quad\mbox{for all}\quad j=1,\ldots,\ell.\]
Let $C^\ast$ denote the adjugate matrix of $C$ and let $d := \det C$; it then follows that
\[ d \omega_i = \sum_{j=1}^\ell C^\ast_{ij} \alpha_j \quad\mbox{for all}\quad i=1,\ldots,\ell.\]
All entries of $C^\ast$ are known to be \emph{non-negative integers}; see either \cite[\S 13.2, Table 1]{Hum} or \cite{LusTits}. Therefore for each $\mu \in \mathbb{N}^\ell$, 
\[ \sum_{i=1}^\ell \mu_i d \omega_i = \sum_{j=1}^\ell \left(\sum_{i=1}^\ell \mu_i C_{ij}^\ast \right) \alpha_j\]
is a linear combination of $\alpha_1,\ldots,\alpha_\ell$ with non-negative integer coefficients. We now observe that for any $\mu \in \mathbb{N}^d$, $\mathfrak{t}_K$ acts on the vector
\[ e_\mu := \prod_{j=1}^\ell f_j^{ \sum_{i=1}^\ell \mu_i C_{ij}^\ast } v_\lambda \in U(\fr{n}_K) \cdot v_\lambda \]
via the character
\[ \lambda - \sum_{j=1}^\ell \left(\sum_{i=1}^\ell \mu_i C_{ij}^\ast \right)\alpha_j = \lambda - \sum_{i=1}^\ell \mu_i d \omega_i.\]
Because our group $\mb{G}$ is simply connected, the coroots $h_i$ span $\fr{t}$ over $R$ and therefore we may identify the affinoid enveloping algebra $\widehat{U(\fr{t})_{n,K}}$ with the Tate algebra $K \langle \pi^n h_1, \dots, \pi^n h_\ell\rangle$. Consider the isomorphism $K\langle x_1,\ldots, x_\ell \rangle \stackrel{\cong}{\longrightarrow} \widehat{U(\fr{t})_{n,K}}$ which sends $x_i$ to $\lambda(\pi^n h_i) - \pi^n h_i$. Viewing $U(\fr{n}_K)\cdot v_\lambda$ as a $K\langle x_1,\ldots,x_\ell\rangle$-module via this isomorphism and remembering that $\omega_i(h_j) = \delta_{ij}$, we can calculate that
\[ x_j \cdot e_\mu = d \pi^n \mu_j e_\mu \quad\mbox{for all}\quad j=1,\ldots, \ell \quad\mbox{and}\quad \mu \in \mathbb{N}^d.\]
Therefore for every $f \in K\langle x_1,\ldots,x_n\rangle$ and $\mu \in \mathbb{N}^d$ we have
\[ f \cdot e_\mu = f(d \pi^n \mu_1, \ldots, d \pi^n \mu_\ell) e_\mu .\]
We may now apply the lemma with each $A_j$ being the infinite subset $d \pi^n \mathbb{N}$ of $R$ to deduce that if $f \in K \langle x_1,\ldots, x_\ell \rangle$ kills every $e_\mu$, then $f = 0$. The result follows.
\end{proof}

\subsection{The Cartan involution}
\label{Cartan}
Recall \cite[\S II.1.4]{Jantzen} that for each element $w$ of the Weyl group $\mb{W}$ of $\mb{G}$ we can find a representative $\dot{w} \in \mb{G}(R)$ of $w \in \mb{W}$ which normalises $\mb{T}(R)$. By \cite[\S II.1.4(4)]{Jantzen}, these elements permute the root subgroups of $\mb{G}$ according to the action of $\mb{W}$ on the root system $\Phi$. If $w_0 \in \mb{W}$ is the longest element, then it follows that 
\[ \dot{w_0} \cdot \fr{b}^+ = \fr{b}\quad\mbx{and}\quad \dot{w_0} \cdot \fr{b} = \fr{b}^+\]
in the adjoint action of $\mb{G}(R)$ on $\fr{g}$.

\begin{prop} $\dot{w_0} \cdot \Ann_{\widehat{U(\fr{b}^+)_{n,K}}} ( \widehat{V^\lambda} )  = \Ann_{\widehat{U(\fr{b})_{n,K}}} ( \widehat{V^\lambda} )$.
\end{prop}
\begin{proof} By Theorem \ref{AnnVerma}, $I := \Ann_{\hugnK} ( \widehat{V^\lambda} )$ is generated by $\ker \chi_\lambda$ which is fixed by the adjoint action of $\mb{G}(R)$ on $\hugnK$. Therefore $\dot{w_0} \cdot I = I$, so
\[ \dot{w_0} \cdot \Ann_{\widehat{U(\fr{b}^+)_{n,K}}} ( \widehat{V^\lambda} ) = \dot{w_0} \cdot I \cap \dot{w_0} \cdot \widehat{U(\fr{b}^+)_{n,K}} = I \cap \widehat{U(\fr{b})_{n,K}} = \Ann_{\widehat{U(\fr{b})_{n,K}}} ( \widehat{V^\lambda} )\]
as required.\end{proof}

\section{Faithfulness of affinoid Verma modules over Iwasawa algebras}
\label{FaithfulIwasawa}

\subsection{$L$-uniform groups}\label{Lunif}
Throughout $\S \ref{FaithfulIwasawa}$ we will assume that $p$ is an odd prime, and that $L$ is a finite extension of $\Qp$ contained in $K$; we have the corresponding chain of inclusions of discrete valuation rings:
\[\Zp \subseteq \O_L \subseteq R.\]
Following Orlik and Strauch \cite[Remark 2.2.5(ii)]{OrlikStrauch}, we say that a uniform pro-$p$ group $G$ is \emph{$L$-uniform} if $G$ is locally $L$-analytic, and the Lie algebra $L_G$ is an $\O_L$-submodule of the $L$-Lie algebra $\mathcal{L}(G)$. The (modified) isomorphism of categories $G \mapsto \frac{1}{p} L_G$ between uniform pro-$p$ groups and finite rank torsion-free $\Zp$-Lie algebras from \cite[Theorem 9.10]{DDMS} induces a one-to-one correspondence between $L$-uniform groups and torsion-free $\O_L$-Lie algebras of finite rank.

Let $G$ be an $L$-uniform group. In $\S\ref{Lunif}$, $\S\ref{Dist1p}$ and $\S\ref{faithful}$ we will suspend the notation from $\S \ref{AffEnvAlg}$ and temporarily use the letter $\fr{g}$ to denote the 
\emph{$R$-Lie algebra associated with $G$}, defined as follows:
\[\fr{g} := R \otimes_{\O_L} \frac{1}{p} L_G.\] 
This extends \cite[Definition 10.2]{AW13} to arbitrary finite extensions $L$ of $\Qp$. Recall that the algebra $D(G,K)$ of $K$-valued locally $L$-analytic distributions from $\S \ref{LocAn}$ is a Frechet-Stein algebra by \cite[Theorem 5.1]{ST}, and therefore we have at our disposal the $K$-Banach space completions $D_r(G,K)$ for each real number $1/p \leq r < 1$.	The abbreviation 
\[\Uhat{g} := \widehat{U(\fr{g})_{0,K}}\]
will also be used throughout $\S \ref{FaithfulIwasawa}$. 

\subsection{The distribution algebra $D_{1/p}(G,K)$}\label{Dist1p}
We begin by recording the following extension of \cite[Theorem 5.1.4]{ArdThesis} and \cite[Proposition 6.10]{SchmidtSmith} to arbitrary complete discrete valuation fields $K$ of mixed characteristic and arbitrary $L$-uniform groups in the language of locally analytic distribution algebras.

\begin{lem} Let $\fr{g}$ be the $R$-Lie algebra associated with the $L$-uniform group $G$. Then there is a natural isomorphism of $K$-Banach algebras
\[ D_{1/p}(G,K) \stackrel{\cong}{\longrightarrow} \Uhat{g}.\]
\end{lem}
\begin{proof} Suppose first that $L = \Qp$. Let $|\cdot| : K \to \mathbb{R}$ be the norm which induces the topology on $K$, normalised by $|p| = 1/p$. Let $g_1,\ldots, g_d$ be a minimal topological generating set for $G$, let $b_i = g_i - 1 \in K[G]$ and write $|\alpha| = \alpha_1 + \cdots + \alpha_d$ for each $\alpha \in \mathbb{N}^d$. It follows from \cite[\S 4]{ST} that the distribution algebra $D_{1/p}(G,K)$ consists of all formal power series $\lambda = \sum_{\alpha \in \mathbb{N}^d} \lambda_\alpha \mathbf{b}^\alpha$ in $b_1,\ldots, b_d$ such that
\[ ||\lambda||_{1/p} := \sup\limits_{\alpha \in \mathbb{N}^d} |d_\alpha| (1/p)^{|\alpha|} \]
is finite. As a consequence of Cohen's Structure Theorem for complete local domains \cite[Corollary 28.P.2]{Matsumura}, we can find an unramified field extension $K'$ of $\Qp$ inside $K$ such that $K/K'$ is finite. Let $R'$ be the ring of integers of $K'$ and let $\fr{g}' = R' \otimes_{\Zp} \frac{1}{p}L_G$. Then we have a commutative diagram
\[\xymatrix{ K \otimes_{K'} D_{1/p}(G,K') \ar[r]\ar[d] & K \otimes_{K'} \widehat{ U(\fr{g}')_{K'} } \ar[d] \\ D_{1/p}(G,K) \ar[r] & \Uhat{g} } \]
where the vertical maps are induced by multiplication inside $D(G,K)$ and $\Uhat{g}$ and the horizontal maps are induced by the inclusion of $G$ into the groups of units of $\widehat{ U(\fr{g}')_{K'} }$ and $\Uhat{g}$, respectively.

Because $K'$ is unramified over $\Qp$, it follows from \cite[Theorem 10.4]{AW13} that the top arrow is an isomorphism. The arrow on the right is a bijection by \cite[Lemma 3.9(c)]{AW13}, and arrow on the left can be seen to be a bijection from the explicit description of elements in $D_{1/p}(G,K)$ as power series in the $b_1,\ldots,b_d$ satisfying the particular convergence condition stated above. The result follows in the case where $L = \Qp$.

Returning to the general case, let $G_0$ be the uniform pro-$p$ group $G$ viewed as a locally $\Qp$-analytic group, and let $\fr{g}_0 := \frac{1}{p}R \otimes_{\Zp} L_G$. Then there are natural surjections of algebras $U(\fr{g}_0) \twoheadrightarrow U(\fr{g})$ and $D_{1/p}(G,K) \twoheadrightarrow D_{1/p}(G_0,K)$, and it follows from \cite[Lemma 5.1]{Schmidt} that
\[ D_{1/p}(G,K) \cong U(\fr{g}) \otimes_{U(\fr{g}_0)} D_{1/p}(G_0,K).\]
Therefore by the first part of this proof we obtain
\[ D_{1/p}(G,K) \cong U(\fr{g}) \otimes_{U(\fr{g}_0)} \widehat{U(\fr{g}_0)_K}.\]
Now the algebra on the right hand side is isomorphic to $\Uhat{g}$ by \cite[\S 3.2.3(iii)]{Berth} because $U(\fr{g}_0)$ is Noetherian.
\end{proof}

\subsection{A general faithfulness result}
\label{faithful}
Let $G$ be an $L$-uniform group with associated $R$-Lie algebra $\fr{g}$. The following result will be our main engine for establishing the faithfulness of modules over the Iwasawa algebra $KG$.

\begin{thm} Let $N$ and $H$ be $L$-uniform subgroups of $G$ with associated $R$-Lie algebras $\fr{n}$ and $\fr{h}$ such that $\fr{g}=\fr{n} \oplus \fr{h}$. Let $V$ be a $\Uhat{g}$-module and suppose that there is $v\in V$ such that $v$ is a free generator of $V$ as a $\Uhat{n}$-module by restriction and $U(\fr{n}_K)v$ is a faithful, locally finite $RH$-module again by restriction. Then $V$ is a faithful $KG$-module. 
\end{thm}
\begin{proof} Let $r$ and $d$ be the ranks of $\fr{n}$ and $\fr{g}$ as $R$-modules, respectively. Choose an $R$-basis $\{x_1,\ldots,x_d\}$ for $\fr{g}$ such that $\{x_1,\ldots,x_r\}$ is an $R$-basis for $\fr{n}$. 

Let $l = [L : \Qp]$; then $G, N$ and $H$ have dimensions $dl, rl$ and $(d-r)l$ respectively when viewed as uniform pro-$p$ groups. We may choose a minimal topological generating set $\{g_1, \ldots, g_{dl}\}$ for $G$ such that $g_1,\ldots, g_{rl}$ and $g_{rl+1},\ldots,g_{dl}$ topologically generate $N$ and $H$, respectively. Write $b_i = g_i-1 \in KG$ for each $i=1,\ldots,ld$.

Suppose that $\zeta\in \Ann_{RG}(V)$. It suffices to prove that $\zeta=0$. We may write $\zeta=\sum_{\alpha\in \mathbb{N}^{dl}}\lambda_\alpha \mathbf{b}^\alpha$ with $\lambda_\alpha\in R$. Collecting terms together we can then rewrite this as $\zeta=\sum_{\alpha\in \mathbb{N}^{rl}} \mathbf{b}^\alpha \zeta_\alpha$ for some $\zeta_\alpha\in RH$. 

Now given $w \in U(\fr{n}_K)v$, $RHw\subset U(\fr{n}_K)v$ is finitely generated over $R$ by assumption. Thus there is a natural number $t$ such that we can write 
\[ \zeta_\alpha\cdot w=\sum_{\beta\in \mathbb{N}^r}\mu_{\beta}^\alpha x^{\beta}\cdot v \] 
for some $\mu_{\beta}^{\alpha}\in K$ with $\mu_{\beta}^{\alpha}=0$ for $|\beta|>t$ and all $\alpha$. Furthermore, we may assume that $\mu_{\beta}^{\alpha}$ is uniformly bounded in $\alpha$ and $\beta$.  Thus 
\[ 0=\zeta\cdot w=\sum_{\alpha\in \mathbb{N}^{rl}} \mathbf{b}^\alpha\zeta_\alpha\cdot w=\sum_{\alpha \in \mathbb{N}^{rl},\beta\in \mathbb{N}^r} \mu^\alpha_\beta \mathbf{b}^{\alpha} x^\beta\cdot v.\] 
But $\sum\limits_{\alpha \in \mathbb{N}^{rl},\beta\in \mathbb{N}^r} \mu^\alpha_\beta \mathbf{b}^{\alpha}x^\beta\in \Uhat{n}$ and $\ann_{\Uhat{n}}(v)=0$ by assumption, so in fact 
\[\sum_{\alpha \in \mathbb{N}^{rl},\beta\in \mathbb{N}^r}\mu^\alpha_\beta \mathbf{b}^{\alpha}x^\beta=0.\]
The multiplication map $KN \otimes_K U(\fr{n}_K) \to D(N,K)$ is injective by Corollary \ref{LocAn}(d). Since $D(N,K)$ contained in $D_{1/p}(N,K)$ which is isomorphic to $\Uhat{n}$ by Lemma \ref{Dist1p}, the multiplication map $KN \otimes_K U(\fr{n}_K) \to \Uhat{n}$ is also injective, and so 
\[\sum_{\beta\in \mathbb{N}^r} \sum_{\alpha\in \mathbb{N}^{rl}} \mu^\alpha_\beta\mathbf{b}^\alpha \otimes  x^{\beta} =0.\]
Therefore $\sum_{\alpha\in \mathbb{N}^{rl}}\mu^{\alpha}_\beta\mathbf{b}^\alpha=0\in KN$ for each $\beta$ because the $x^{\beta}$ are linearly independent over $K$. It follows that $\mu^\alpha_\beta=0$ for each pair $\alpha,\beta$, and hence $\zeta_{\alpha}\cdot w=0$ for each $\alpha$.  As this last equality is independent of the choice of $w\in U(\fr{n}_K)v$, we deduce that $\zeta_\alpha\in \Ann_{RH}(U(\fr{n})v)=0$ for each $\alpha$ and hence $\zeta  = \sum_{\alpha \in \mathbb{N}^{rl}} \mb{b}^\alpha \zeta_\alpha = 0$.
\end{proof}

\subsection{Congruence kernels}
\label{AffKGfaithful}
We now assume that the $L$-uniform group $G$ and the $R$-algebraic group $\mathbf{G}$ from $\S\ref{AdU}$ satisfy the following conditions:
\begin{itemize}
\item $\mb{G}$ is simply connected,
\item  the Lie algebra $\fr{g}$ of $\mb{G}$ and the Lie algebra $L_G$ of $G$ satisfy $p^n \fr{g} = \frac{1}{p}R \otimes_{\O_L} L_G$ for some integer $n \geq 0$,
\item $p$ is a very good prime for $\mb{G}$ in the sense of \cite[\S 6.8]{AW13}.
\end{itemize}
For example, for every integer $n \geq 0$, $G$ could be the congruence kernel
\[G = \ker(\mb{G}(\O_L) \to \mb{G}(\O_L/p^{n+1}\O_L)).\]
As in $\S\ref{AffVerma}$, we let $\fr{t}, \fr{b}, \fr{n}$ and $\fr{b}^+$ denote the Lie algebras of $\mb{T}$, $\mb{B}$, $\mb{N}$ and $\mb{B}^+$ of $\mb{G}$ respectively and note that because these groups are defined over $\O_L$ we can find $L$-uniform subgroups $T, B, N$ and $B^+$ whose respective associated $R$-Lie algebras are $p^n\fr{t}, p^n\fr{b}, p^n\fr{n}$ and $p^n	\fr{b}^+$.

\begin{thm} The affinoid Verma module $\widehat{V^\lambda}$ is faithful as a $KG$-module for every $R$-linear character $\lambda : p^n \fr{t} \to R$.
\end{thm}
\begin{proof} Since $p^n \fr{n}$ is a complement to $p^n \fr{b}^+$ in $p^n \fr{g}$, $\widehat{V^\lambda}$ is a free $\widehat{U(p^n\fr{n})_K}$-module of rank $1$ generated by the highest weight vector $v \in \widehat{V^\lambda}$. The dense submodule $U(\fr{n}_K)\cdot v$ of $\widehat{V^\lambda}$ is locally finite as a $\fr{b}^+_K$-module; this implies that it is also a locally finite $RB^+$-submodule of $\widehat{V^\lambda}$. In particular, it is a locally finite $RT$-module.

Since $RT$ is a subring of $\widehat{U(p^n\fr{t})_K}$, Proposition \ref{AdRT} implies that the action of $RT$ on $U(\fr{n}_K)\cdot v$ is faithful. Since $p^n \fr{b} = p^n \fr{n} \oplus p^n\fr{t}$, $\widehat{V^\lambda}$ is faithful as a $RB$-module by Theorem \ref{faithful}. 

Proposition \ref{Cartan} now implies that $\widehat{V^\lambda}$ is also faithful as a $RB^+$-module, so its submodule $U(\fr{n}_K)\cdot v$ is also faithful over $RB^+$. Since $p^n \fr{g} = p^n \fr{n} \oplus p^n \fr{b}^+$, invoking Theorem \ref{faithful} again gives that $\widehat{V^\lambda}$ is faithful as a $KG$-module, as required.
\end{proof}

\subsection{Verma modules for congruence kernels}
\label{VermaCong}
For each locally $L$-analytic character $\theta\colon B^+\to 1+pR$, the contragredient of the natural $1$-dimensional representation given by $\theta$ induces a continuous $1$-dimensional $D(B^+,K)$-module $K_\theta$ via \cite[Corollary 3.4]{ST4}. We may view $K_\theta$ as a $KB^+$-module by restriction recovering the original $1$-dimensional representation. That is $b\cdot x=\theta(b)x$ for $b\in B^+$ and $x\in K_\theta$. 

By instead restricting along the inclusion $U(\fr{b}^+_K)\to D(B^+,K)$ we may view $K_\theta$ as a representation $\lambda$ of the $K$-Lie algebra $\fr{b}^+_K$. We can compute that for $b\in B^+$, $\log b\in p^{n+1}\mathfrak{b}^+$ acts by $\log \theta(b)\in pR$. Thus $\lambda$ may be viewed as an $R$-linear character $p^n\fr{b}^+\to R$. 

Conversely, each $R$-linear map $\lambda : p^n \fr{b}^+ \to R$ induces a locally $L$-analytic homomorphism $\theta : B^+ \to 1+pR$ via the rule
\[ b \mapsto \exp(\lambda(\log b)) \quad\mbox{for all} \quad b \in B^+. \]

\begin{defn} Let $\theta : B^+ \to 1 + pR$ be a locally $L$-analytic group homomorphism. The \emph{Verma module} for the Iwasawa algebra $KG$ with highest weight $\theta$ is 
\[ M^\theta = KG \otimes_{KB^+} K_\theta\]
where $K_\theta $ is the one-dimensional $KB^+$-module $K$ with $B^+$-action given by $\theta$.
\end{defn}

\begin{lem} Let $\lambda : p^n \fr{b}^+ \to R$ be the $R$-linear character  of $p^n \fr{b}^+$ corresponding to a locally $L$-analytic group homomorphism $\theta : B^+ \to 1 + pR$. Then the $KG$-submodule of $\widehat{V^\lambda}$ generated by the highest weight vector $v_\lambda$ is naturally isomorphic to $M^\theta$.
\end{lem}
\begin{proof} By construction, $B^+$ acts on $v_\lambda \in \widehat{V^\lambda}$ via $\theta$. Sending the highest weight vector $m_\theta \in M^\theta$ to $v_\lambda$ therefore induces a $KG$-module homomorphism $M^\theta \longrightarrow \widehat{V^\lambda}$, which fits into the following commutative diagram:
\[ \xymatrix{ KN \ar[r]\ar[d] & \widehat{U(p^n\fr{n})_K} \ar[d] \\ M^\theta \ar[r] & \widehat{V^\lambda}. } \]
Here the vertical arrows are bijections that send $x \in KN$ to $xm_\theta$ and $y \in \widehat{U(p^n\fr{n})_K}$ to $yv_\lambda$, respectively. The top arrow is the natural inclusion of $KN$ into $\widehat{U(p^n\fr{n})_K}$, so $M^\theta \longrightarrow \widehat{V^\lambda}$ is injective. The result follows.
\end{proof}

\begin{cor} $M^\theta$ is a faithful $KG$-module.
\end{cor}
\begin{proof} The commutative diagram in the proof of the lemma shows that the image of $M^\theta$ is dense in $\widehat{V^\lambda}$. Now if an element of $KG$ kills $M^\theta$, then it must annihilate all of $\widehat{V^\lambda}$ by continuity, and is therefore zero by Theorem \ref{AffKGfaithful}. \end{proof}

\subsection{Finite normal subgroups}\label{FNS}

Before we can give a proof of Theorem A, we will need to understand better the finite normal subgroups of open subgroups of $\mb{G}(\O_L)$. 

\begin{prop} Let $G$ be an open subgroup of $\mb{G}(\O_L)$ and let $F$ be a finite normal subgroup of $G$. Then $F$ is central in $\mb{G}(\O_L)$.
\end{prop}
\begin{proof} Choose a torsion-free open subgroup $N$ of $G$ which is normal in $\mb{G}(\O_L)$, for example a congruence kernel of $\mb{G}(\O_L)$. Then $[N,F] \leq N \cap F = 1$ because $N$ and $F$ are both normal in $G$ and because $N$ is torsion-free, so $F$ centralises $N$. Because $\O_L$ is local, $\mb{G}(\O_L)$ is generated as an abstract group by the elements of the form $x_\alpha(r)$ from $\S\ref{AffEnvAlg}$ for $\alpha\in \Phi$ and $r \in \O_L$ by \cite[Proposition 1.6]{Abe69}, so it will be sufficient to show that $F$ commutes with each $x_\alpha(r)$. 

Fix $g \in F$, $\alpha \in \Phi$ and $r \in \O_L$, let $x:= x_\alpha(r)$ and note that $g$ commutes with $x^t$ where $t$ is the index of $N$ in $\mb{G}(\O_L)$. Choose a faithful algebraic representation $\rho : \mb{G}(\overline{L}) \to GL_m(\overline{L})$ for some $m \geq 1$, let $u_1 = \rho(x)$ and $u_2 = \rho( gxg^{-1} )$. Then $u_1$ and $u_2$ are unipotent by \cite[Theorem 15.4(c)]{Hum3} and $u_1^t = u_2^t$ because $g$ commutes with $x^t$. Now $\log$ and $\exp$ give well-defined bijections between unipotent matrices in $GL_m(\overline{L})$ and nilpotent matrices in $M_m(\overline{L})$, so 
\[u_1 = \exp\left(\frac{1}{t}\log(u_1^t)\right) = \exp\left(\frac{1}{t}\log(u_2^t)\right) = u_2.\]
Therefore $x = gxg^{-1}$ because $\rho$ is faithful.
\end{proof}

\subsection{Proof of Theorem A}\label{PfA}
By continuity, we can find an open subgroup $H$ of $B^+$ which is mapped into $1 + pR$ by $\theta$. Choose $n$ large enough so that $G$ contains an open normal $L$-uniform subgroup $G_n$ with associated $R$-Lie algebra $p^n \fr{g}$ and such that $H$ contains an open $L$-uniform subgroup $B^+_n$ with associated $R$-Lie algebra $p^n \fr{b}^+$. Let $\theta_n$ be the restriction of $\theta$ to $B^+_n$. Then since $B^+_n=B^+\cap G_n$,
\[ M^{\theta_n} = KG_n \otimes_{KB^+_n} K_{\theta_n} \subseteq KG \otimes_{KB^+} K_\theta.\]
Writing $I := \Ann_{KG} (KG \otimes_{KB^+} K_\theta)$ and applying Corollary \ref{VermaCong} gives
\[ I \cap KG_n \subseteq \Ann_{KG_n} M^{\theta_n} = 0.\]
Let $F$ be a finite normal subgroup of $G$. Then $F$ is central in $\mb{G}(\O_L)$ by Proposition \ref{FNS}, so $F$ is also central in $G$. But $G$ has trivial centre by assumption so $F$ is trivial. Therefore $KG$ is a prime ring by \cite[Theorem A]{ArdBro2007}.

Next, $KG$ is a crossed product of $KG_n$ with the finite group $G/G_n$. By \cite[Theorem 2.1.15]{MCR}, $S := KG_n \backslash \{0\}$ is an Ore set in the Noetherian domain $KG_n$. It is stable under conjugation by $G$, so by \cite[Lemma 37.7]{Pass} it is also an Ore set in the larger ring $KG$ and there is a crossed product decomposition
\[ S^{-1} KG = (S^{-1} KG_n) \ast (G / G_n).\]
Here $S^{-1} KG_n$ is the quotient division ring of fractions of $KG_n$; thus $S^{-1} KG$ is an Artinian ring because the group $G/G_n$ is finite. Every regular element in $KG$ stays regular in $KG_n$ and therefore becomes invertible in $S^{-1}KG$ by \cite[Proposition 3.1.1]{MCR}; hence $S^{-1}KG$ is the classical Artinian ring of quotients of $KG$. 

Since $I \cap KG_n = 0$, the intersection $I \cap S$ is empty and therefore the two-sided ideal $S^{-1} I$ of $S^{-1} KG$ is proper. Now $KG$ is prime, so $S^{-1} KG$ is a prime Artinian ring and is therefore \emph{simple}. Therefore $S^{-1}I = 0$ and $I = 0$. \qed

\bibliography{references}
\bibliographystyle{plain}
\end{document}